\documentclass[letterpaper, DIV=17]{scrartcl}

\usepackage{pgfplots}
\pgfplotsset{compat=1.18}

\usepackage{amssymb, amsthm, mathtools,bbm, bm}
\usepackage[square,sort,comma,numbers]{natbib}
\usepackage{comment}

\usepackage[none]{hyphenat}
\usepackage[english]{babel}
\usepackage{enumerate}
\usepackage{amsmath}
\usepackage{graphicx}%
\usepackage{subcaption}
\usepackage{amsthm}
\usepackage{amssymb}
\usepackage{dsfont}%
\usepackage{float}%
\usepackage{hyperref}
\usepackage{amssymb}
\usepackage{color}
\usepackage[utf8]{inputenc}
\usepackage{graphicx}
\usepackage{tikz}
\usepackage{tikz-cd}
\usepackage{verbatim} 
	\usepackage{mathtools} 
	\usepackage{mathabx} 
	\usepackage{hyperref}
	\usepackage{comment}
	\usetikzlibrary{decorations.markings}
	\usepackage{bbm} 
	\usepackage{xfrac} 
	\usepackage{tensor} 
	\usepackage{listings}  
	\usepackage{matlab-prettifier} 
	
	\usepackage{pgfplots}              
	\pgfplotsset{compat=newest}
	\usepgfplotslibrary{fillbetween}

	\usepackage[title]{appendix}

\usepackage[shortcuts]{extdash}  
	
	\newtheorem{theorem}{Theorem}[section]  
	\newtheorem{cor}[theorem]{Corollary}
	\newtheorem{prop}[theorem]{Proposition}
	\newtheorem{lemma}[theorem]{Lemma}

	\numberwithin{equation}{section}
	\theoremstyle{definition}
	\newtheorem{defn}{Definition}[section]
	\newtheorem{rem}[theorem]{Remark}

\newcommand{\btau}{\boldsymbol{\tau} }
\newcommand{\bsig}{\boldsymbol{\sigma} }

\newcommand{\beq}[1]{\begin{equation} \label{#1}}
	\newcommand{\eeq}{\end{equation}}

\renewcommand{\epsilon}{\varepsilon}

\newcommand{\eps}{\varepsilon}

\def\be{\begin{equation}}
	\def\ee{\end{equation}} 

\DeclareMathOperator\arctanh{arctanh}

\begin{document}
		\title{Lower bound on the mixing time of $p$-spin glasses}
		
		\author{Anouar Kouraich$^{1}$ and  Simone Warzel$^{1,2,3}$ \\
			\\
			\small $^1$ Department of Mathematics, TU Munich, Garching, Germany \\[-.5ex]
			\small $^2$ Munich Center for Quantum Science and Technology, Munich, Germany \\[-.5ex]
			\small $^3$ Department of Physics, TU Munich, Garching, Germany}

        \date{\small \today \\[-.5cm]  }
		\maketitle		
		
		\begin{abstract}
			We show that Glauber dynamics for $ p$-spin glass mixes exponentially slowly at inverse temperatures larger than a constant times $ \ln (p)/p $ for large enough $ p $. This is done by analyzing the energy landscape using Gaussian decompositions and establishing a bottleneck bound.
		\end{abstract}

\section{Introduction}
\subsection{Main result}
Glauber dynamics corresponding to inverse temperature $ \beta $ and a system's Hamiltonian $ H : \{ -1,1\}^N \to \mathbb{R} $, $ \bsig \mapsto H(\bsig) $ on the configuration space of $ N $ Ising spins is an irreducible, aperiodic Markov chain 
designed such that the Gibbs distribution
\begin{equation}
    \pi(\bsig) \coloneqq Z_N(\beta)^{-1} \exp\left(- \beta H(\bsig)\right) , \quad Z_N(\beta) \coloneqq \sum_{\bsig \in \{-1,1\}^N} \exp\left(- \beta H(\bsig)\right) ,
\end{equation}
is its unique stationary distribution. For Glauber dynamics, the transition matrix $ P(\bsig,\btau) $ characterizing the chain
has the property
\begin{equation}\label{eq:walk}
    P(\bsig,\btau) = 0 \quad \mbox{unless $\displaystyle  d(\bsig,\btau) \coloneqq \sum_{j=1}^N \frac{|\sigma_j - \tau_j|}{2}  \leq 1 $, }
\end{equation} 
reflecting the fact that 
in each time step, one flips at most one spin such that the Hamming distance $ d(\bsig,\btau) $ between the initial configuration $ \bsig $ and the final one $ \btau $ is at most $1$. 
The specific choice $ P(\bsig,\btau) = N^{-1} \left( 1+ \exp\left[\beta (H(\btau) - H(\bsig) )\right] \right)^{-1} $ in case $ d(\bsig,\btau) = 1$ and accordingly $ P(\bsig,\bsig) = 1- \sum_{\btau :\  d(\bsig,\btau) = 1} P(\bsig,\btau)$ corresponds to the standard Glauber Markov chain for a mean-field Hamiltonian such as the the Ising spin glass with $ p $-spin interaction 
\begin{equation}\label{eq:pspin}
    H(\bsig) \coloneqq - \frac{1}{N^{(p-1)/2}} \sum_{j_1, \dots , j_p = 1 }^N g_{j_1, \dots , j_p} \sigma_{j_1} \cdots \sigma_{j_p} .
\end{equation}
The latter is given in terms of independent and identically distributed standard normal variables $ g_{j_1, \dots , j_p}$. The energies hence form a 
Gaussian random process on $ \{ -1 , 1\}^N $ with mean zero and covariance, 
\begin{equation}\label{eq:covariance}
\mathbb{E}[H(\bsig) H(\btau) ] = N  \left( 1 - \frac{d(\bsig,\btau)}{2N} \right)^p =: N \ c_p\left( \frac{d(\bsig,\btau)}{N}\right) , 
\end{equation}
which only depends on the Hamming distance $ d $. 
The special case $ p = 2 $ corresponds to the SK model \cite{SK75}. Taking the limit $ p \to \infty $ in \eqref{eq:covariance} one arrives at the simplest Ising spin glass, the random energy model (REM), in which the energies $ H(\bsig) $ are uncorrelated. 

Spin glasses are prototypes of complex energy landscapes. One measure of their complexity is the time it takes to prepare their Gibbs distribution via Glauber dynamics. For arbitrary initial probability distribution $ \mu $ on $ \{ -1,1\}^N $, this time is characterized by the mixing time
\begin{equation}
t_{\textrm{mix}} \coloneqq \min\left\{ t \ \Big| \ \sup_{\mu} \frac{1}{2} \sum_{\btau} \left| \mu P^t(\btau) - \pi(\btau) \right| \leq \frac{1}{4} \right\} , 
\end{equation}
where $ \mu P^t $ denotes the evolved distribution at time $ t \in \mathbb{N} $. The latter is the $ t $-fold convolution of the transition kernel:
$$
\mu P^t(\btau) \coloneqq \sum_{\bsig_1 \dots \bsig_t} \mu(\bsig_t) P(\bsig_t,\btau) P(\bsig_{t-1}, \bsig_t) \dots P(\bsig_1,\bsig_2) .
$$

We establish lower bounds on the mixing time, summarized in the following main result.
\begin{theorem}\label{thm:main}
Consider any irreducible, aperiodic Markov chain, which has the Gibbs distribution of the $ p $-spin glass~\eqref{eq:pspin} as its stationary distribution and satisfies~\eqref{eq:walk}. 
There are constants $ c, C \in (0,\infty) $ and $ p_0 \in \mathbb{N} $ such that  for all $ p \geq p_0 $ and all $ \displaystyle\beta > C \  \frac{\ln p}{p}$:
\begin{equation}
   \lim_{N \to \infty } \mathbb{P}\left(  t_{\textrm{mix}} \geq e^{N c} \right) = 1 . 
\end{equation}
\end{theorem}
The proof is spelled out in Section~\ref{sec:proof}.
\subsection{Discussion}
For Markov chains with a positive transition matrix, e.g.\ Glauber dynamics \cite{DGU14}, the mixing time is linked to the spectral gap $1- \lambda_2 $ between the largest eigenvalue $ 1 $ of $ P $ and its second largest $ \lambda_2 $ via the estimates~\cite[Theorems 12.4 \& 12.5]{LY17}
\begin{equation}\label{eq_spectral_gap_mixing_bound}
\ln (2)\left(\frac{1}{1- \lambda_2}-1\right) \leq t_{\textrm{mix}}  \leq \frac{\ln 4 - \ln \min_{\bsig} \pi(\bsig)}{1- \lambda_2} .
\end{equation}
Since $\ln \min_{\bsig} \pi(\bsig) \geq C N$, the above result shows that with asymptotically full probability, the spectral gap is exponentially small. \\

\noindent
The Ising $p$-spin glass is expected to exhibit the following sequence of dynamical and static phase transitions, summarized in Figure~\ref{fig:beta-phases}.
\begin{figure}[H]
  \centering
  \begin{tikzpicture}[>=stealth]
    \draw[->] (0,0) -- (16,0) node[right] {$\; \beta$};

    \foreach \x/\lab in {
        0/0,
        3/$\beta_{\text{sl}}(p)$,
        8/$\beta_{\text{sh}}(p)$,
        12/$\beta_{\text{st}}(p)$
    }{
      \draw (\x,0) -- (\x,-0.3) node[below] {\lab};
    }

    \node[below=0.7cm, align=center] at (1.5,0)
      {\emph{fast mixing}};

    \node[above=0.5cm, align=center] at (5.5,0)
      {\emph{slow mixing from}\\\emph{worst-case initialization}};

    \node[below=0.7cm, align=center] at (5.5,0)
      {\emph{fast mixing from}\\\emph{random initialization}};

    \node[above=0.5cm, align=center] at (10,0)
      {\emph{slow mixing from}\\\emph{random initialization}};

    \node[above=0.5cm, align=center] at (14,0)
      {\emph{spin glass phase}};
\end{tikzpicture}
  \caption{Conjectured hierarchy of transition temperatures for the Ising $p$-spin glass. Theorem~\ref{thm:main} establishes the upper bound  $ \beta_{\text{sl}}(p) \leq C (\ln p )/p $.}
  \label{fig:beta-phases}
\end{figure}
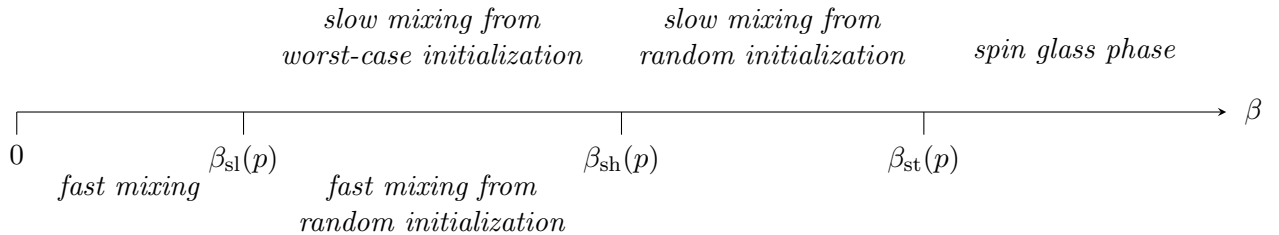

\bigskip 
\noindent
The model undergoes a static spin-glass transition at the inverse temperature $\beta_{\mathrm{st}}(p)$. For $\beta>\beta_{\mathrm{st}}(p)$, replica symmetry breaking (RSB) occurs. For large $p$,
\[
    \beta_{\mathrm{st}}(p)=\beta_c\bigl(1-o_p(1)\bigr),\qquad
    \beta_c\coloneqq\sqrt{2\ln 2},
\]
where $o_p(1)\to0$ as $p\to\infty$ \cite{Bov06,G85,Z24}.
Gardner \cite{G85} predicted a further transition for $p\ge 3$ at a lower temperature, where the system passes from a one-step RSB phase (recently established in \cite{Z24}) to a full RSB phase. For even $p\geq 4$, it was shown in \cite{BJ18} that the Glauber dynamics mix exponentially slowly throughout the spin-glass regime.

\bigskip
\noindent
For $p\geq 3$, the Ising $p$-spin glass is expected to undergo a first dynamical transition at $\beta_{\mathrm{sh}}(p)$. For large $p$,
\[
    \beta_{\mathrm{sh}}(p)=\sqrt{\frac{2\ln p}{p}}\,\bigl(1+o_p(1)\bigr),
\]
see \cite{KT87,FLPR12,MR03}; a short derivation of the above expression appears in \cite[Appendix~A]{AMS25_sampling}.
For inverse temperatures $\beta\in\bigl(\beta_{\mathrm{sh}}(p),\beta_{\mathrm{st}}(p)\bigr),$
the Gibbs measure is predicted to be \emph{shattered}: it concentrates on exponentially many clusters, each carrying an exponentially small fraction of the total mass.
This geometry implies exponential slowdown for the Glauber dynamics from worst initialization, as the clusters form a natural bottleneck for the dynamics. Even more, it is widely believed that the Glauber dynamics from a random initialization mix exponentially slowly throughout this regime, but a proof remains open.

The existence of the shattering phase for large $p$ was first proven rigorously by Gamarnik, Jagannath, and Kızıldağ, and later refined (near-optimally) by El Alaoui \cite{Al25}, matching the physics predictions \cite{FLPR12,MR03} asymptotically.
Slow mixing of Glauber dynamics and shattering has also been established for mixed Ising spin glasses in \cite{BJ18,AAS25}.

The model is believed to undergo a second dynamical transition at some  $\beta_{\mathrm{sl}}(p)$: Glauber dynamics still mix slowly from worst-case initial data. From random initialization, the dynamics, however, already mix efficiently in this regime. Theorem~\ref{thm:main} addresses the exponentially slow mixing beyond the shattering regime, and establishes an explicit upper bound on the breakdown of fast mixing:
\[
\beta_{\mathrm{sl}}(p)\leq \frac{C\ln p}{p}
\]
for some constant $C>0$. This should be contrasted with known regimes of fast mixing at sufficiently high temperatures: in \cite{ABX22}, a spectral gap of order $1/N$ was proven for $\beta<C/\sqrt{p^3\ln p}$. In this regime, the standard estimate \eqref{eq_spectral_gap_mixing_bound}implies $t_{\mathrm{mix}}=O(N^2)$. This was improved in \cite{AJK23}, which establishes a modified log-Sobolev inequality (MLSI) with constant of order $1/N$ in the same temperature regime, giving $t_{\mathrm{mix}}=O(N\ln N)$. 
Determining the precise value of $\beta_{\mathrm{sl}}(p)$ or, showing that the Glauber dynamics from random initialization mix efficiently for $\beta<\beta_{\mathrm{sh}}(p)$ and slowly for $\beta>\beta_{\mathrm{sh}}(p)$, remains an open problem.

\bigskip
\noindent
\textbf{Related models and other sampling dynamics.}
For the Sherrington--Kirkpatrick (SK) model ($p=2$), fast mixing was proven for $\beta<\beta_{\mathrm{st}}(2)/4$ by establishing a spectral gap \cite{EKZ22}, by building on the breakthrough \cite{BB19}. In the same regime, \cite{AnEA21} obtained an optimal mixing-time bound $O(N\ln N)$ by establishing an MLSI. More recently, \cite{AKV24} extended the fast-mixing regime to all $\beta<\beta_{\mathrm{st}}(2)/4+c$, where $c>0$ is a small absolute constant. It is conjectured that fast mixing holds throughout the entire replica-symmetric phase, i.e., for all $\beta<\beta_{\mathrm{st}}(2)$. 
At sufficiently low temperatures, Sellke recently showed slow mixing \cite{S25}.\\

Beyond Glauber dynamics, several recent works \cite{AMS25_sampling,DLSS26,HMP24,HMRW24} analyze alternative sampling algorithms and prove efficient sampling of the Gibbs measure in temperature regimes that extend beyond the best known for Glauber dynamics.\\

The Langevin dynamics of the spherical $p$-spin glass exhibits behavior closely analogous to that of the Ising $p$-spin glass with transition temperatures with different values, but lined up as in Figure~\ref{fig:beta-phases}. The physical picture underlying the dynamical transition was first described in \cite{CHS93,CK93}. Gheissari and Jagannath \cite{GJ19} proved the existence of a spectral gap at sufficiently high temperature, and an exponentially small spectral gap at sufficiently low temperature. Slow mixing in, and beyond, the spin-glass phase was established by Ben Arous and Jagannath \cite{BJ18,BJ24} for $p \geq 4$. More recently, the emergence of shattering was proven in \cite{AMS25_spherical} for large $p$.
Closely related to the present work is  \cite[Theorem~2.6]{BJ24}, which shows that, for $\beta > T_{\mathrm{BBM}}^{-1}(p) =\frac{e}{\sqrt{2\ln (p)}}(1+o_p(1))$, the Langevin dynamics mixes exponentially slowly even in the absence of shattering, as predicted in \cite{BBM96}; see~\cite[Equation A.37]{BBM96}.

\section{Proof}\label{sec:proof}
This paper was inspired by the recent work \cite{S25} by Sellke, which concerns the analogue of Theorem~\ref{thm:main} for $ p = 2 $. Its proof \cite{S25} relies on techniques \cite{DGZ25,HS25}, which are not (yet) available for general $p$. We circumvent this by exploiting a characterization of the energy landscape of the $ p$-spin glass directly using  Gaussian decompositions. The latter control correlations effectively for large $ p $. Such techniques have previously been used in the context of quantum spin glasses~\cite{MW20,MW21b,MW25} in \cite{kmw26}.  
\subsection{Deterministic bottleneck bound}

The proof is based on the Jerrum-Sinclair-Diaconis-Stroock bottleneck bound \cite[Theorem 7.4]{LY17}:
\begin{equation}\label{eq:JSDS}
    \frac{1}{4\  t_{\textrm{mix}} } \leq \min_{\substack{A \subset \{ -1,1\}^N \\ \pi(A) \leq 2^{-1}}} \frac{1}{\pi(A)} \sum_{\substack{\bsig \in A \\ \btau \not\in A}} \pi(\bsig) P(\bsig,\btau)  
\end{equation}
In our situation, we will work with balls and their surfaces
$$
B_r(\bsig_0) \coloneqq \left\{ \bsig \ | \ d(\bsig , \bsig_0) \leq r N \right\} , \qquad S_r(\bsig_0) \coloneqq \left\{ \bsig \ | \ d(\bsig , \bsig_0) =  r N \right\} 
$$
of radius $ r \in (0,2^{-1}) $. 
Provided $ \bsig_0 $ is chosen such that $ \pi(B_r(\bsig_0) ) \leq 2^{-1}$, we may further estimate the right side of \eqref{eq:JSDS} using~\eqref{eq:walk} and the normalization $ 1 = \sum_{\btau} P(\bsig,\btau) $:
\begin{align}
    \frac{1}{\pi(B_r(\bsig_0))} \sum_{\substack{\bsig \in B_r(\bsig_0) \\ \btau \not\in B_r(\bsig_0)}} \pi(\bsig) P(\bsig,\btau) &=\frac{1}{\pi(B_r(\bsig_0))} \sum_{\substack{\bsig \in S_r(\bsig_0) \\ \btau \not\in B_r(\bsig_0)}} \pi(\bsig) P(\bsig,\btau)  \leq  \frac{\pi(S_r(\bsig_0))}{\pi(B_r(\bsig_0))}  \notag \\
    & \leq \left| S_r(\bsig_0)\right| \ \exp\left(\beta \left[ H(\bsig_0) - \min_{\btau \in S_r(\bsig_0)} H(\btau) \right] \right) \notag \\
    & \leq \exp\left( N \gamma(r) + \beta \left( H(\bsig_0) (1- c_p(r) ) - \min_{\btau \in S_r(\bsig_0)} G_{\bsig_0}(\btau) \right] \right) . \label{eq:JSDS2}
\end{align} 
Here and in the following $ | \cdot | $ denotes the cardinality of a set, such that by the standard bound on the binomial coefficient in terms of the binary entropy function $\gamma $:
\begin{equation}\label{eq:binom}
  \left| S_r(\bsig_0)\right| = \binom{N}{rN} \leq e^{ N \gamma(r) } , \qquad \gamma(r) \coloneqq - r \ln r - (1-r) \ln(1-r) .
\end{equation}
A key ingredient in the last line of~\eqref{eq:JSDS2} was the Gaussian decomposition:
\begin{equation}\label{eq:Gdec}
   G_{\bsig_0}(\btau) \coloneqq  H(\btau) -  H(\bsig_0) \ c_p\left( \frac{d(\bsig_0,\btau)}{N}\right)  . 
\end{equation}
To exploit the deterministic bound~\eqref{eq:JSDS2}, we 
will pick $ \bsig_0$ from the set of large negative fluctuations,
\begin{equation}\label{def:L}
    \mathcal{L}_{\varepsilon} \coloneq \left\{ \bsig \ | \ - H(\bsig) > \beta_c (1- \varepsilon) N \right\} , 
\end{equation}
with $  \varepsilon \in (0,1) $. 
Our deterministic bound will hold on the following event.
\begin{defn}\label{def:G}
For $ \varepsilon , \delta  \in (0,1)$ and $ r \in [N^{-1}, 2^{-1}) $, we define the event $ \mathcal{G}_{\varepsilon , \delta}(r) $ by requiring:
\begin{enumerate}
    \item $ | \mathcal{L}_\varepsilon | > | B_{2r} | $, with the latter the volume of any ball with radius $ 2 r N $, and
    \item\label{p2} For all $ \bsig_0 \in \mathcal{L}_\varepsilon  $: \quad $ \displaystyle \min_{\btau \in S_r(\bsig_0)} G_{\bsig_0}(\btau) \geq -N \beta_c (1-\varepsilon) \delta (1-c_p(r) ) $. 
\end{enumerate}
\end{defn}
The following summarizes our deterministic bound, and shows that for $ \beta $ large enough on the above event, the mixing time is exponentially large. 
\begin{lemma}\label{lem:bottlesum}
On the event $ \mathcal{G}_{\varepsilon , \delta}(r) $ there is some $ \bsig_0 \in \mathcal{L}_\varepsilon $ such that $ \pi\left( B_r(\bsig_0) \right)\leq 2^{-1} $, and one has:
\begin{equation}\label{eq:JSDSfinal}
 4 \  t_{\textrm{mix}} \geq \exp\left( N \left[ \beta  \beta_c (1-\varepsilon) (1-\delta)(1-c_p(r) )  - \gamma(r) \right]\right) . 
\end{equation}  
\end{lemma}
\begin{proof}
Since $| \mathcal{L}_\varepsilon | > |B_{2r}| \geq 2 $, there are two distinct $ \bsig_0 , \btau_0 \in \mathcal{L}_\varepsilon $ such that $ B_r(\bsig_0) \cap B_r(\btau_0) = \emptyset$. Hence
$$
\pi(B_r(\bsig_0) ) + \pi(B_r(\btau_0)) \leq 1 ,
$$
from which we conclude that at least one of the two terms is smaller than $ 1/2$. 

The bound~\eqref{eq:JSDSfinal} is then an immediate consequence of \eqref{eq:JSDS}, \eqref{eq:JSDS2} and the characterizing property~\ref{p2} of the event $\mathcal{G}_{\varepsilon , \delta}(r) $. 
\end{proof}

\subsection{Probabilistic bounds}

 For a fixed $ \bsig \in \{ -1 , 1\}^N $, the probability of an  event of the form featuring in~\eqref{def:L} is given by the error integral 
$$ \mathbb{P}(- H(\bsig) > u \sqrt{N} ) = \int_{u  }^\infty  e^{-x^2/2} \frac{dx}{\sqrt{2\pi}} =: \Phi(u )  ,
$$
for which we will frequently use the following standard estimates
\begin{equation}\label{eq:errorint}
\frac{1}{\sqrt{2\pi} u } \left(1- u^{-2}\right) e^{- u^2/2} \leq \Phi(u) \leq \frac{1}{\sqrt{2\pi} u } e^{- u^2/2}
\end{equation}
 valid for any $ u > 0 $. The estimate in particular implies that on average, the set $ \mathcal{L}_\varepsilon $ is exponentially large, i.e., for any $ \varepsilon \in (0,1) $:
\begin{equation}\label{eq:Llower}
\mathbb{E}\left[| \mathcal{L}_{\varepsilon}  | \right] = 2^N \Phi\left(\sqrt{N} \beta_c (1-\varepsilon)\right) \geq \frac{1}{\sqrt{2\pi N} \beta_c (1-\varepsilon)} \left( 1 - \frac{1}{\beta_c^2 (1-\varepsilon)^2 N} \right) \exp\left(N \frac{\beta_c^2}{2} \varepsilon (2 - \varepsilon) \right) 
\end{equation}
More importantly, 
\cite{GJK25} showed that this expectation is faithful via a standard second moment method.
\begin{prop}[\text{\cite[Proposition 2.1]{GJK25}}]\label{prop_concentration}
For any $ \varepsilon \in (0,1) $ there is some $ p(\varepsilon) \in \mathbb{N} $ such that for all $ p \geq p(\varepsilon)$:
\begin{equation}
    \lim_{N\to \infty} \frac{\mathbb{E}\left[| \mathcal{L}_{\varepsilon}  | \right]^2 }{\mathbb{E}\left[| \mathcal{L}_{\varepsilon}  |^2 \right]} = 1 . 
\end{equation}
\end{prop}
\begin{rem}
In the argument of \cite{GJK25}, it is necessary to choose the parameter $p(\varepsilon) > C\varepsilon^{-1}$ for some constant $C>0$.
\end{rem}
An immediate corollary is the following bound on the probability of the first event in Definition~\ref{def:G}.
\begin{cor}\label{cor:P1}
    For any $ \varepsilon \in (0,1) $ and $ r \in (0,4^{-1}) $ such that
    $
        \gamma(r) \leq 4^{-1} \beta_c^2  \varepsilon $, 
    there is some $ c_{\varepsilon}(r) > 0 $ and $ N_\varepsilon(r) \in \mathbb{N} $ such that for all $ N \geq N_\varepsilon(r) $ 
    \begin{equation}\label{eq:ratio}
    \theta_\varepsilon(r) \coloneqq \frac{|B_{2r}|}{\mathbb{E}\left[| \mathcal{L}_{\varepsilon}  | \right] } \leq e^{- c_{\varepsilon}(r) N } .
    \end{equation}
    Moreover, one also has
    \begin{equation}\label{eq:remprob}
        \mathbb{P}\left( |\mathcal{L}_\varepsilon | > |B_{2r}| \right) \geq \left(1-\theta_\varepsilon(r)\right)^2 \ \frac{\mathbb{E}\left[| \mathcal{L}_{\varepsilon}  | \right]^2 }{\mathbb{E}\left[| \mathcal{L}_{\varepsilon}  |^2 \right]} . 
    \end{equation}
\end{cor}
\begin{proof}
    For an estimate on the ratio in~\eqref{eq:ratio}, we use the monotonicity of the binomial coefficients for $ 2r \leq 2^{-1} $ and ~\eqref{eq:binom} to conclude
    $$
        |B_{2r}| = \sum_{j=0}^{2rN} \binom{N}{j} \leq (2rN+1) \binom{N}{2rN} \leq (2rN+1) e^{N\gamma(2r) } \leq (2rN+1) e^{2N \gamma(r)},
    $$
    where the last step is by concavity of $ \gamma $.
    The bound~\eqref{eq:ratio} hence follows from~\eqref{eq:Llower}, since $2\gamma(r)\leq 2^{-1}\beta_c^2\eps$

   Since $  \theta_\varepsilon(r) \in (0,1) $, the estimate~\eqref{eq:remprob} is simply the Paley-Zygmund inequality:
   $$
    \mathbb{P}\left( |\mathcal{L}_\varepsilon | > |B_{2r}| \right) = \mathbb{P}\left( |\mathcal{L}_\varepsilon | > \theta_\varepsilon(r) \ \mathbb{E}\left[| \mathcal{L}_{\varepsilon}  | \right]\right) \geq ( 1- \theta_\varepsilon(r) )^2 \  
    \frac{\mathbb{E}\left[| \mathcal{L}_{\varepsilon}  | \right]^2 }{\mathbb{E}\left[| \mathcal{L}_{\varepsilon}  |^2 \right]} , 
    $$
    which completes the proof.
\end{proof}

It remains to establish a bound on the second event in Definition~\ref{def:G}. This will heavily rely on the fact that for any  $ \bsig_0 $ and any $ \btau $ such that $ d(\bsig_0 , \btau ) = r N $ the random variables $ G_{\bsig_0}(\btau) =  H(\btau) -  H(\bsig_0) \ c_p\left( r\right) $ are Gaussian with mean zero and variance 
\begin{equation}
\mathbb{E}\left[ G_{\bsig_0} (\btau)^2\right] = N \left[  1 - c_p(r)^2 \right], \quad \mbox{if $ d(\bsig_0,\btau) = r N $.} 
\end{equation} 
Moreover, they are uncorrelated,
$
\mathbb{E}\left[ H(\bsig_0) \  G_{\bsig_0} (\btau)\right] = 0 
$, 
and hence independent of $ H(\bsig_0) $.
\begin{lemma}\label{lem:P2}
    For any $ \varepsilon, \delta \in (0,1) $ and $ r \in (0,2^{-1}) $:
    \begin{multline}\label{eq:P2}
        \mathbb{P}\left( \mbox{There is $ \bsig_0 \in \mathcal{L}_\varepsilon$}: \quad 
        \min_{\btau \in S_r(\bsig_0)} G_{\bsig_0}(\btau) < -N \beta_c (1-\varepsilon) \delta (1-c_p(r) ) \right) \\
        \leq \frac{\sqrt{1+c_p(r)}}{2\pi N \beta_c^2 (1-\varepsilon)^2 \delta \sqrt{1-c_p(r)}}\  \exp\left( - \frac{N}{2} \left[ \beta_c^2\left(  (1-\varepsilon)^2 \delta^2 \frac{1-c_p(r)}{1+c_p(r)} - 2 \varepsilon \left(1 - \frac{\varepsilon}{2}\right) \right) - 2 \gamma(r)  \right] \right).
    \end{multline}
\end{lemma}
\begin{proof}
    We use the union bound to estimate the left side of~\eqref{eq:P2} by
    \begin{align}
         \sum_{\bsig_0 \in \{ -1,1\}^N }& \mathbb{P}\left( \bsig_0 \in \mathcal{L}_\varepsilon \; \mbox{and} \; \min_{\btau \in S_r(\bsig_0)} G_{\bsig_0}(\btau) <- N \beta_c (1-\varepsilon) \delta (1-c_p(r) ) \right) \notag \\
        & \leq  \sum_{\bsig_0 \in \{ -1,1\}^N } \mathbb{P}\left( \bsig_0 \in \mathcal{L}_\varepsilon\right) \sum_{\btau \in S_r(\bsig_0)} \mathbb{P}\left(G_{\bsig_0}(\btau) \leq -N \beta_c (1-\varepsilon) \delta (1-c_p(r) )\right) \notag \\
        & \leq 2^N \Phi(\sqrt{N} \beta_c (1-\varepsilon))\  \left| S_r(\bsig_0) \right| \Phi\left( \sqrt{N} \beta_c (1-\varepsilon) \delta \frac{1-c_p(r) }{\sqrt{1 - c_p(r)^2}} \right) . 
    \end{align}
    The assertion then follows from~\eqref{eq:binom} and the standard estimate~\eqref{eq:errorint} on the error integral. 
\end{proof}

\subsection{Proof of Theorem~\ref{thm:main}}

\begin{proof}
    Let $ \delta\in (0,1) $ be arbitrary and pick $ \varepsilon \in (0, 1 ) $ small enough such that 
    $$ x_{\varepsilon,\delta} \coloneqq \frac{3\varepsilon}{(1-\varepsilon)^2\delta^2} < 1 . 
    $$
    We now pick 
    \begin{equation}\label{def:r}
        r_{\varepsilon,\delta}(p) \coloneqq \frac{1}{2p} \ln \frac{1+x_{\varepsilon,\delta} }{1- x_{\varepsilon,\delta}} = \frac{\arctanh(x_{\varepsilon,\delta})}{p} .
    \end{equation}
    This ensures that 
    $$ c_p(r_{\varepsilon,\delta}(p)) \leq e^{-2p \ r_{\varepsilon,\delta}(p)} = \frac{1-x_{\varepsilon,\delta} }{1+ x_{\varepsilon,\delta}} 
    $$
    and hence
    $$
    (1-\varepsilon)^2 \delta^2 \ \frac{1-c_p(r_{\varepsilon,\delta}(p))}{1+c_p(r_{\varepsilon,\delta}(p))} \geq (1-\varepsilon)^2 \delta^2  x_{\varepsilon,\delta} = 3 \varepsilon  . 
    $$
    In case $ p \geq p_{\varepsilon,\delta} $ with $ p_{\varepsilon,\delta} \in \mathbb{N} $ large enough such that $ r_{\varepsilon,\delta}(p) < 4^{-1} $ and 
    \begin{equation}\label{eq:upper_binary_r}
        \gamma(r_{\varepsilon,\delta}(p)) \leq \frac{\beta_c^2 \varepsilon}{4} ,
    \end{equation}
    the exponent in the right side of \eqref{eq:P2} is non-positive and hence the probability converges to zero as $ N \to \infty $. With these choices of parameters, Corollary~\ref{cor:P1} and Lemma~\ref{lem:P2} imply that
    \begin{multline*}
    \mathbbm{P}\left( \mathcal{G}_{\varepsilon,\delta}(r_{\varepsilon,\delta}(p))^c \right)\leq 1 - \mathbb{P}\left( |\mathcal{L}_\varepsilon | > |B_{2r_{\varepsilon,\delta}(p)}| \right)\\  + \mathbb{P}\left( \mbox{There is $ \bsig_0 \in \mathcal{L}_\varepsilon$:} \quad 
        \min_{\btau \in S_{r_{\varepsilon,\delta}(p)}(\bsig_0)} G_{\bsig_0}(\btau) < -N \beta_c (1-\varepsilon) \delta (1-c_p(r_{\varepsilon,\delta}(p)) ) \right)    
    \end{multline*}
    converges to zero as $ N \to \infty $. However, on $ \mathcal{G}_{\varepsilon,\delta}(r_{\varepsilon,\delta}(p)) $, Lemma~\ref{lem:bottlesum} implies an exponential bound on the mixing time provided
    \begin{equation}\label{eq:cutoff_beta}    
    \beta > \frac{\gamma(r_{\varepsilon,\delta}(p)) \ ( 1 + x_{\varepsilon,\delta}) }{ \beta_c (1-\varepsilon)(1-\delta) 2 x_{\varepsilon,\delta} } .
    \end{equation}
    
    Since $$\gamma(r_{\varepsilon,\delta}(p)) \leq \frac{2 \arctanh( x_{\varepsilon,\delta}) }{p} \ln \frac{ p }{ \arctanh( x_{\varepsilon,\delta}) } , $$ this concludes the proof. 
\end{proof}\noindent
\textbf{On an optimal value for $p_0$:}
Our main restriction on $p$ comes from inequality \eqref{eq:upper_binary_r}. For small $\eps$, this condition takes the form
\[
p_{\eps,\delta}\sim \frac{\ln(1/\eps)\,\arctanh(x_{\eps,\delta})}{\eps}.
\]
A suitable choice of $\delta$ can reduce this bound substantially. This is especially valuable for even \(p\), where one can use a symmetry argument similar to that of \cite{S25} to avoid using Proposition \ref{prop_concentration}, and hence eliminate the lower bound \(p(\eps) \ge C/\eps\). However, since we require the set \(\mathcal{L}_\eps\) to be nonempty, for smaller values of $p$ we must take \(\eps\) sufficiently large.

We can also optimize the lower bound in \eqref{eq:cutoff_beta} with respect to $\eps$ and $\delta$ in the regime of asymptotically large $p$. In that case, the lower bound is of the form
\[
C_{\eps,\delta}\frac{\ln p}{p}(1+o_p(1)):=\frac{\arctanh(x_{\eps,\delta})(1+x_{\eps,\delta})}{2\beta_c(1-\eps)(1-\delta)x_{\eps,\delta}}\,
\frac{\ln p}{p}(1+o_p(1)).
\]
Since $\arctanh(x)\geq x$ for all $x\geq 0$, we obtain
$
C_{\eps,\delta}\geq 
\frac{1+x_{\eps,\delta}}{2\beta_c(1-\eps)(1-\delta)}
\geq 
\frac{1}{2\beta_c} $. 
For small $\eps$ and $\delta$, this lower bound is essentially sharp. Indeed, choose
$\delta:=(6\eps)^{1/3}$. Then
$x_{\eps,\delta}=\left(\frac{3}{4}\eps\right)^{1/3}(1+O(\eps))$,
and hence
\[
\frac{1}{2\beta_c}
\leq
\min_{0<\eps,\delta,\;x_{\eps,\delta}<1} C_{\eps,\delta}
\leq
\frac{1}{2\beta_c}\left(1+\left(\frac{3}{4}\eps\right)^{1/3}+O(\eps^{2/3})\right).
\]

\paragraph{Acknowledgements.} The Deutsche Forschungsgemeinschaft funded this work as part of DFG–TRR 352\--Project-ID 470903074. 

\bibliography{bibliography}{}

@book{Bov06,
	address = {Cambridge},
	author = {Bovier, A.},
	db = {Cambridge Core},
	dp = {Cambridge University Press},
	pages = {i-vi},
	publisher = {Cambridge University Press},
	series = {Cambridge Series in Statistical and Probabilistic Mathematics},
	title = {Statistical Mechanics of Disordered Systems: A Mathematical Perspective},
    year={2010}
}

@article{GJK25,
title = "Shattering in the {I}sing p-spin glass model",
author = "D. Gamarnik and A. Jagannath and E. C. Kızıldağ ",
year = "2025",
volume = "193",
pages = "89--141",
journal = "Probability Theory and Related Fields",
publisher = "Springer New York",
number = "1-2",
}

@book{LY17,
  author    = {Levin, D. A. and Peres, Y.},
  title     = {Markov Chains and Mixing Times},
  publisher = {American Mathematical Society},
  year      = {2017},
  edition   = {2nd}
}

@article{BJ18,
  title={Spectral gap estimates in mean field spin glasses},
  author={Ben Arous, G. and Jagannath, A.},
  journal={Communications in Mathematical Physics},
  volume={361},
  number={1},
  pages={1--52},
  year={2018},
  publisher={Springer}
}

@article{ABX22,
      title={Spectral Gap Estimates for Mixed $p$-Spin Models at High Temperature}, 
      author={A. Adhikari and C. Brennecke and C. Xu and H.-T. Yau},
      journal={Probab. Theory Relat. Fields},
      volume={189},
      pages={879--907},
      year={2024}, 
}

@unpublished{S25,
      title={Exponentially Slow Mixing of the Low Temperature SK Model}, 
      author={M. Sellke},
      year={2025},
      note={arXiv:2511.22621},
      archivePrefix={arXiv},
      primaryClass={math.PR}, 
}

@unpublished{DGZ25,
      title={Maximally-stable Local Optima in Random Graphs and Spin Glasses: Phase Transitions and Universality}, 
      author={Y. Dandi and D. Gamarnik and L. Zdeborová},
      year={2025},
    note={arXiv:2305.03591},
      archivePrefix={arXiv},
      primaryClass={math.PR}, 
}

@unpublished{HS25,
      title={Strong Low Degree Hardness for Stable Local Optima in Spin Glasses}, 
      author={B. Huang and M. Sellke},
      year={2025},
    note={arXiv:2501.06427},
      archivePrefix={arXiv},
      primaryClass={cond-mat.dis-nn},
}

@article{DGU14,
   title={Structure and eigenvalues of heat-bath Markov chains},
   volume={454},
   journal={Linear Algebra and its Applications},
   publisher={Elsevier BV},
   author={Dyer, M. and Greenhill, C. and Ullrich, M.},
   year={2014}, pages={57–71} }

@article{BJ24,
author = {Ben Arous, G. and Jagannath, A.},
title = {Shattering versus metastability in spin glasses},
journal = {Communications on Pure and Applied Mathematics},
volume = {77},
number = {1},
pages = {139-176},
year = {2024}
}

@article{AMS25_sampling,
   title={Sampling from mean-field Gibbs measures via diffusion processes},
   volume={6},
   number={3},
   journal={Probability and Mathematical Physics},
   publisher={Mathematical Sciences Publishers},
   author={A. {El Alaoui} and A. Montanari and M. Sellke},
   year={2025},
    pages={961–1022} }

@article{Al25,
title = {Near-optimal shattering in the {Ising} pure p-spin and rarity of solutions returned by stable algorithms},
journal = {Stochastic Processes and their Applications},
volume = {192},
pages = {104792},
year = {2026},

author = {A. {El Alaoui}},
}

@article{AMS25_spherical,
author = {{El Alaoui}, Ahmed and Montanari, Andrea and Sellke, Mark},
year = {2025},
pages = {},
title = {Shattering in Pure Spherical Spin Glasses},
volume = {406},
journal = {Communications in Mathematical Physics},
}

@article{FLPR12,
  title = {Two-step relaxation next to dynamic arrest in mean-field glasses: Spherical and {I}sing $p$-spin model},
  author = {Ferrari, U. and Leuzzi, L. and Parisi, G. and Rizzo, T.},
  journal = {Phys. Rev. B},
  volume = {86},
  issue = {1},
  pages = {014204},
  numpages = {8},
  year = {2012},
  publisher = {American Physical Society},
}

@unpublished{Z24,
      title={On the {G}ardner Transition in the {I}sing Pure $p$-Spin Glass}, 
      author={Y. Zhou},
      year={2024},
      note={arXiv:2408.14630},
      archivePrefix={arXiv},
      primaryClass={math.PR},
}

@article{G85,
    author = "Gardner, E.",
    title = "{Spin glasses with p spin interactions}",
    reportNumber = "Print-85-0710 (EDINBURGH)",
    journal = "Nucl. Phys. B",
    volume = "257",
    pages = "747--765",
    year = "1985"
}

@article{MR03,
   title={On the nature of the low-temperature phase in discontinuous mean-field spin glasses},
   volume={33},
   number={3},
   journal={The European Physical Journal B - Condensed Matter},
   publisher={Springer Science and Business Media LLC},
   author={Montanari, A. and Ricci-Tersenghi, F.},
   year={2003},
   month={June}, pages={339–346} }

@article{BB19,
   title={A very simple proof of the LSI for high temperature spin systems},
   volume={276},
   number={8},
   journal={Journal of Functional Analysis},
   publisher={Elsevier BV},
   author={Bauerschmidt, R. and Bodineau, T.},
   year={2019},
   month={Apr}, pages={2582–2588} }

@article{EKZ22,
author = {Eldan, R. and Koehler, F. and Zeitouni, O.},
year = {2022},
month = {04},
pages = {},
title = {A spectral condition for spectral gap: fast mixing in high-temperature {I}sing models},
volume = {182},
journal = {Probability Theory and Related Fields},
}

@article{AnEA21,
  author    = {N. Anari and
               V. Jain and
               F. Koehler and
               H. T. Pham and
               T.{-}D. Vuong},
  title     = {Entropic Independence {I}: Modified Log-{S}obolev Inequalities for Fractionally Log-Concave Distributions and High-Temperature {Ising} Models},
  journal   = {CoRR},
  volume    = {abs/2106.04105},
  year      = {2021},
  eprinttype = {arXiv},
  eprint    = {2106.04105},
}

@inproceedings{AJK23,
author = {N. Anari and V. Jain and F. Koehler and H. T. Pham and T.-D. Vuong},
title = {Universality of Spectral Independence with Applications to Fast Mixing in Spin Glasses},
booktitle = {Proceedings of the 2024 Annual ACM-SIAM Symposium on Discrete Algorithms (SODA)},
year={2024},
chapter = {},
pages= {5029--5056},
}

@unpublished{AAS25,
      title={On the Discontinuous Breaking of Replica Symmetry and Shattering in Mean-Field Spin Glasses}, 
      author={A. Auffinger and A. {El Alaoui} and M. Sellke},
      year={2025},
      note={arXiv:2506.02238},
      archivePrefix={arXiv},
      primaryClass={math.PR}, 
}

@article{GJ19,
   title={On the spectral gap of spherical spin glass dynamics},
   volume={55},
   number={2},
   journal={Annales de l’Institut Henri Poincaré, Probabilités et Statistiques},
   publisher={Institute of Mathematical Statistics},
   author={Gheissari, R. and Jagannath, A.},
   year={2019},
   month={May} }

@unpublished{HMP24,
      title={Sampling from Spherical Spin Glasses in Total Variation via Algorithmic Stochastic Localization}, 
      author={B. Huang and A. Montanari and H. T. Pham},
      year={2024},
      note={arXiv:2404.15651},
      archivePrefix={arXiv},
      primaryClass={math.PR}, 
}

@article{BBM96,
   title={Dynamics within metastable states in a mean-field spin glass},
   volume={29},
   number={5},
   journal={Journal of Physics A: Mathematical and General},
   publisher={IOP Publishing},
   author={Barrat, A. and Burioni, R. and Mézard, M.},
   year={1996},
   month={Mar}, pages={L81–L87} }

@ARTICLE{CHS93,
	author = {Crisanti, A. and Horner, H. and Sommers, H.-J.},
	title = {The spherical p-spin interaction spin-glass model - The dynamics},
	year = {1993},
	journal = {Zeitschrift für Physik B Condensed Matter},
	volume = {92},
	number = {2},
	pages = {257 – 271},
	type = {Article},
	publication_stage = {Final},
	source = {Scopus},
}

@article{CK93,
   title={Analytical solution of the off-equilibrium dynamics of a long-range spin-glass model},
   volume={71},
   number={1},
   journal={Physical Review Letters},
   publisher={American Physical Society (APS)},
   author={Cugliandolo, L. F. and Kurchan, J.},
   year={1993},
   month= {July}, pages={173–176} }

@article{KT87,
  title = {p-spin-interaction spin-glass models: Connections with the structural glass problem},
  author = {Kirkpatrick, T. R. and Thirumalai, D.},
  journal = {Phys. Rev. B},
  volume = {36},
  issue = {10},
  pages = {5388--5397},
  numpages = {0},
  year = {1987},
  month = {Oct},
  publisher = {American Physical Society},
}

@unpublished{kmw26,
      title={The Quantum Random Energy Model is the Limit of Quantum $ p $-Spin Glasses}, 
      author={A. Kouraich and C. Manai and S. Warzel},
      year={2026},
      note={arXiv:2505.02458. To appear in: Markov Processes and Related Fields}
}

@article{MW20,
	author = {Manai, C. and Warzel, S.},
	journal = {Journal of Statistical Physics},
	number = {1},
	pages = {654--664},
	title = {Phase Diagram of the Quantum Random Energy Model},
	volume = {180},
	year = {2020}
	}

@article{MW25,
	author = {C. Manai and S. Warzel},
	journal = {Electronic Journal of Probability},
	pages = {1 -- 39},
	publisher = {Institute of Mathematical Statistics and Bernoulli Society},
	title = {{A Parisi formula for quantum spin glasses}},
	volume = {30},
	year = {2025}}

@article{MW21b,
  author  = {C. Manai and S. Warzel},
  title   = {The Quantum Random Energy Model as a Limit of $p$-Spin Interactions},
  journal = {Reviews in Mathematical Physics},
  volume  = {33},
  number  = {01},
  pages   = {2060013},
  year    = {2021}
}

@inproceedings{AKV24,
author = {Anari, N. and Koehler, F. and Vuong, T.-D.},
title = {Trickle-Down in Localization Schemes and Applications},
year = {2024},
publisher = {Association for Computing Machinery},
address = {New York, NY, USA},
booktitle = {Proceedings of the 56th Annual ACM Symposium on Theory of Computing},
pages = {1094–1105},
numpages = {12},
location = {Vancouver, BC, Canada},
series = {STOC 2024}
}

@article{SK75,
  title = {Solvable Model of a Spin-Glass},
  author = {Sherrington, D. and Kirkpatrick, S.},
  journal = {Phys. Rev. Lett.},
  volume = {35},
  issue = {26},
  pages = {1792--1796},
  numpages = {0},
  year = {1975},
  publisher = {American Physical Society},
}

@unpublished{DLSS26,
      title={Potential Hessian Ascent {III}: Sampling the {S}herrington--{K}irkpatrick Model at Beta {$<$} 1/2}, 
      author={E. Davies and H. Lee and J. S. Sandhu and J. Shi},
      year={2026},
      note={arXiv:2605.03718},
      archivePrefix={arXiv},
      primaryClass={math.PR},
}

@inproceedings{HMRW24,
author = {Huang, B. and Mohanty, S. and Rajaraman, A. and Wu, D. X.},
title = {Weak Poincar\'{e} Inequalities, Simulated Annealing, and Sampling from Spherical Spin Glasses},
year = {2025},
publisher = {Association for Computing Machinery},
address = {New York, NY, USA},
booktitle = {Proceedings of the 57th Annual ACM Symposium on Theory of Computing},
pages = {915–923},
numpages = {9},
location = {Prague, Czechia},
series = {STOC '25}
}
\bibliographystyle{plainnat}
\end{document}